\documentclass[a4paper,12pt]{article}
\usepackage[subpreambles=true]{standalone}
\usepackage{diss}
\usepackage[onehalfspacing]{setspace}

\begin{document}

\title{Uniqueness of First Passage Time Distributions via Fredholm Integral Equations}

\author{Sören Christensen\thanks{Kiel University, \emph{Email: christensen@math.uni-kiel.de}}, Simon Fischer\thanks{Kiel University, \emph{Email: fischer@math.uni-kiel.de}}, Oskar Hallmann\thanks{Kiel University, \emph{Email: o.hallmann@math.uni-kiel.de}}}
\date{\today}
\maketitle

\begin{abstract}
Let $W$ be a standard Brownian motion with $W_0 = 0$ and let $b: \rr_+ \to \rr$ be a continuous function with $b(0) > 0$. The first passage time (from below) is then defined as
\begin{align*}
\tau := \inf \{ t \geq 0 \vert W_t \geq b(t) \}.
\end{align*}
It is well-known that the distribution $F$ of $\tau$ satisfies a set of Fredholm equations of the first kind, which is used, for example, as a starting point for numerical approaches. 
For this, it is fundamental that the Fredholm equations have a unique solution. In this article, we prove this in a general setting using analytical methods. 
\end{abstract}

 \vspace{.8cm}

\textbf{Subject Classifications AMS MSC 2020:} 60J65; 45B05; 60G40

\section{The First Passage Time Problem and Fredholm Intgral Equations}
Let $W$ be a standard Brownian motion with $W_0 = 0$ and let $b: \rr_+ \to \rr$ be a continuous function with $b(0) > 0$. The first passage time (from below) is then defined as
\begin{align*}
\tau := \inf \{ t \geq 0 \vert W_t \geq b(t) \}.
\end{align*}
The first passage time (FPT) problem seeks to determine the distribution $F$ of $\tau$, i.e., $F(t) = P(\tau \leq t)$. 
The FPT problem has a long history and can for example be traced back as far as joint work by Kolmogorov and Khintchine which was later published in \cite{Khi33}. However, closed form solutions for $F$ (given $b$) are rare. Apart from the well known Bachelier-L\'evy formula for linear boundaries, there are closed solutions for the square-root boundary (see e.g., \cite{Bre67}, \cite{RSS84} and \cite{NFK99}) and for the quadratic boundary in terms of Airy functions (see e.g., \cite{Sal88} and \cite{Gro89}). One method for producing analytically solvable examples is the method of images, see \cite{Ler86}, but its use for treating given barriers is limited, see e.g., \cite{LRD02}, \cite{Zip13}. 

Another approach to the FPT problem is by integral equations (typically Volterra and Fredholm integral equations) which connect $b$ and $F$. In \cite{Pes02} the author presents a generalised approach to derive Volterra integral equations. While these equations are in most cases difficult to solve analytically, there are numerical approaches to solve them (see e.g., \cite{Dur71}, \cite{Smi72}) and semi-closed forms for the solutions which are still hard to compute explicitly except in special cases (see e.g., \cite{PP74}). 

Fredholm type integral equations have also been used but more rarely. In \cite{She67} and \cite{Nov81}, Fredholm type integral equations helped to formulate known integral transformations of $F$ and are used to determine moments of $\tau$ as well as asymptotic behaviour of $F$. In \cite{Dan00} Fredholm type integral equations are used to derive FPT distributions for perturbations of linear boundaries. \cite{JKV09} generalises the approach from \cite{Pes02} to link Volterra and Fredholm equations characterising $F$. And while \cite{JKV09} gave results when Volterra type integral equations determine $F$ uniquely, no uniqueness results for Fredholm type integral equations were found. Recently, \cite{CMW19} concerned themselves with numerical solutions to Fredholm integral equations and applied these to the FPT problem. But to the best of our knowledge, the problem whether Fredholm type integral equations uniquely determine $F$ has not been treated.

We denote the transition kernel of  $W$ by $p$ and assume $b:\rr^{+}\to \rr$ to be continuously differentiable with linear growth, i.e., $\limsup_{t\to \infty}\frac{b(t)}{t}\leq d$ for some $d \in \RR$. This implies that there exist $M,m>0$ such that $b(t)\leq M+mt$. There are different approaches to establishing a Fredholm-type integral equation for the distribution $F$ of $\tau$. A direct way is as follows: 
 Using the strong Markov property, we obtain
\begin{align}
p((0,0),(t,x)) &= P(W_t\in \d x) \nonumber \\
&= \int_{0}^{t}P_{(s,b(s))}(W_t\in \d x) P(\tau\in \d s) + P(W_t\in \d x, \tau>t).
\end{align}
Since $b$ is continuously differentiable, $P(\tau\in \d s)$ has a continuous density function $f$ (see e.g., \cite{Fer82}). Then,
\begin{align}
1 & = \frac{p((0,0),(t,x))}{p((0,0),(t,x))} 
= \frac{\int_{0}^{t}P_{(s,b(s))}(W_t\in \d x)f(s)\d s}{p((0,0),(t,x))} 
+\frac{P(W_t\in \d x, \tau>t)}{p((0,0),(t,x))}.
\end{align}
For $c > d$ we set $x = c t$ and let $t\to \infty$. The last term can be reformulated as
\begin{align*}
\frac{P(W_t\in\d (ct), \tau>t)}{p((0,0),(t,ct))} 
& = P(\tau>t\mid W_t = ct)\xrightarrow{t\to \infty} 0 
\end{align*}
since $\limsup_{t\to \infty}\frac{b(t)}{t}\leq d$ and $c > d$. On the other hand, using linear growth of $b$ and the dominated convergence theorem, the first term can be rewritten as
\begin{align*}\lim_{t\to \infty}\frac{\int_{0}^{t}P_{(s,b(s))}(W_t\in \d (ct))f(s)\d s}{p((0,0),(t,ct))}  &= \lim_{t\to \infty}\int_{0}^{t}\frac{p((s,b(s)),(t,ct))f(s)}{p((0,0),(t,ct))}\d s\\
&= \int_{0}^{\infty}e^{-\frac{c^2}{2}s + c b(s)}f(s)\d s.
\end{align*}
Combining this, we obtain
\begin{equation}\label{eq:hit1}
1 = \int_{0}^{\infty}e^{-\frac{c^2}{2}s + c b(s)}f(s)\d s
 \end{equation}
for all $c>d$ such that the integral exists. A visualization of the setting can be seen in Figure \ref{fig:setting1}.
    \begin{figure}[htbp]
        \centering
        \includegraphics[width=0.7\textwidth]{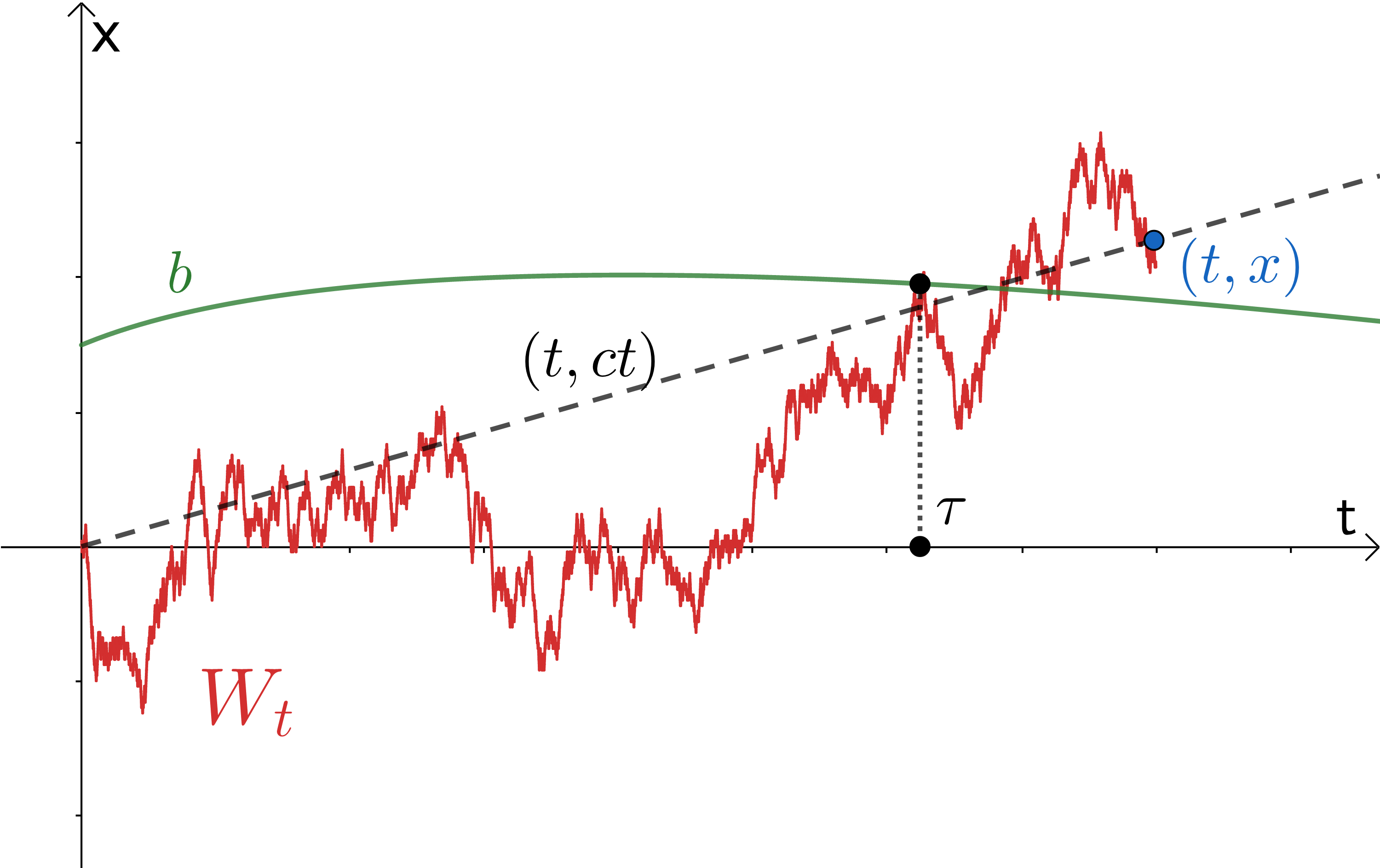}
        \caption{Visualization of the setting}
        \label{fig:setting1}
    \end{figure}
Now, the task is: For a given boundary $b$ we would like to find $f$ with $\int_{0}^{\infty}f(s)\d s = P(\tau < \infty)$, such that \eqref{eq:hit1} holds for all $c$.

\section{Uniqueness}
A central question is if solutions to \eqref{eq:hit1} are unique, i.e., if \eqref{eq:hit1} fully characterizes the distribution of $\tau$. Our ansatz is inspired by recent results for the Fredholm representation for discounted stopping problems with finite time horizon in \cite{CF21}. We use the following notation
for $J\subset (-\infty,0]$ measurable and $n\in\nn_{0}$:
\begin{align*}
    & I^{f}_{J}(n)(c) := \int_{J}e^{-\frac{c^2}{2}s + c b(s)}s^{n}f(s)\d s\\
    & I^{f}(n)(c): = I^{f}_{[0,\infty)}(n)(c)\\
    & I^{f}_{J}(c): = I^{f}_{J}(0)(c),
\end{align*}
and for two different continuous functions $f$ and $\tilde f$
\begin{align*}
   & D_{J}(n)(c) := I^{f}_{J}(n)(c) - I^{\tilde f}_{J}(n)(c)\\ 
   & D(n)(c) := D_{[0,\infty)}(n)(c).
\end{align*}

\begin{lemma}\label{lem:hit1}
     Let $t\geq 0$ be fixed and $\varepsilon >0$, then
        \[\lim_{c\to \infty}\frac{I_{[t,\infty)}^{f}(n)(c)}{I^{f}_{[t,t+\varepsilon]}(n)(c)} = 1.\]
    If additionally $f(t) \neq \tilde f(t)$, then
        \[\lim_{c\to \infty}\frac{D_{[t,\infty)}(n)(c)}{D_{[t,t+\varepsilon]}(n)(c)} = 1.\]
\end{lemma}
\begin{proof}
    Let $(c_i)$ be a sequence such that $c_i \to \infty$ for $i\to \infty$. We first show that for all $\varepsilon >0$ and $t\geq 0$ it holds
    \begin{equation}\label{eq:i1}
           \lim_{i\to \infty}\frac{I_{[t,\infty)}^{f}(n)(c_i)}{I^{f}_{[t,t+\varepsilon]}(n)(c_i)} = 1.
    \end{equation}  
    Equation \eqref{eq:i1} is equivalent to
        \[\lim_{i\to \infty}\frac{I^{f}_{[t+\varepsilon,\infty)}(n)(c_i)}{I	^{f}_{[t,t+\varepsilon]}(n)(c_i)} = 0.\]
    The numerator and denominator are positive. Let us derive an upper bound for $I^{f}_{[t+\varepsilon,\infty)}(n)(c)$.  Since $f$ is continuous and integrable we can set $f_1 = \max\{f(s) \mid s\in[t +\varepsilon,\infty)\}$. Furthermore, we use that $s^{n}e^{-s}\leq n^{n}e^{-n}$. For $c$ large enough we have
    \begin{align*}
        I^{f}_{[t+\varepsilon,\infty)}(n)(c)
        & =\int_{t+\varepsilon}^{\infty}e^{-\frac{1}{2}c^2s+c b(s)}s^{n} f(s)\d s\\
        & \leq n^{n}e^{-n} f_1\int_{t+\varepsilon}^{\infty}e^{-\frac{1}{2}c^2s+c (M+m s)+s}\d s \\
        & = \frac{n^{n}e^{-n} f_1}{\frac{c^2}{2}-1-cm}e^{(t+\varepsilon)(cm - \frac{c^2}{2}-1)+ cM}.
    \end{align*}
    We now derive a lower bound for $I^{f}_{[t,t+\varepsilon]}(n)(c_i)$. The function $b$ is continuous, so we can set $b_1 : = \min \{b(s) \mid s\in[t+\frac{\varepsilon}{2} ,t+\varepsilon]\}$. Furthermore, let $f_2 : = \min \{f(s) \mid s\in[t+\frac{\varepsilon}{2} ,t+\varepsilon]\}$. Note that $f_2>0$, since $b$ is locally Lipschitz continuous. We have
     \begin{align*}
        I^{f}_{[t,t+\varepsilon]}(n)(c)
        & =\int_{t}^{t+\varepsilon}e^{-\frac{1}{2}c^2s+c b(s)}s^{n} f(s)\d s\\
        & \geq f_2\left(t+\frac{\varepsilon}{2}\right)\int_{t+\frac{\varepsilon}{2}}^{t+\varepsilon}e^{-\frac{1}{2}c^2s+c b_1}\d s\\
        & = f_2\left(t+\frac{\varepsilon}{2}\right)\frac{2}{c^2}\left(e^{-\frac{c^2}{2}\left(t+\frac{\varepsilon}{2}\right) + cb_1} - e^{-\frac{c^2}{2}(t+\varepsilon) + cb_1}\right).
    \end{align*}
    Putting these results together we find
    \begin{align*}
        0&\leq \lim_{i\to \infty}\frac{I^{f}_{[t+\varepsilon,\infty)}(n)(c_i)}{I^{f}_{[t,t+\varepsilon]}(n)(c_i)}\\
        &\leq \lim_{i\to \infty} \frac{\frac{n^{n}e^{-n} f_1}{\frac{c_i^2}{2}+1-cm}e^{(t+\varepsilon)(c_im - \frac{c_i^2}{2}-1)+ c_iM}}{f_2\left(t+\frac{\varepsilon}{2}\right)\frac{2}{c_i^2}\left(e^{-\frac{c_i^2}{2}\left(t+\frac{\varepsilon}{2}\right) + c_ib_1} - e^{-\frac{c_i^2}{2}(t+\varepsilon) + c_ib_1}\right)} = 0
    \end{align*}   
    which shows the first claim. \\
    Let now $f(t) \neq \tilde f(t)$. We assume w.l.o.g.\ that $f(t) > \tilde f(t)$. Since $f$ and $\tilde f$ are continuous, there exists $\delta>0$ such that $f(s) > \tilde f(s)$ for all $s\in [t,t+\delta]$.
    Then, the term
        \[D_{[t,t+\delta]}(n)(c_i) = \int_{t}^{t+\delta}e^{-\frac{1}{2}c^2s+c b(s)}s^{n} (f(s) - \tilde f(s))\d s\\\]
    is positive and analogously to the calculations above we obtain 
        \[ \lim_{i\to \infty}\frac{|D_{[t+\delta, \infty)}(n)(c_i)|}{D_{[t,t+\delta]}(n)(c_i)} 
        \leq  \lim_{i\to \infty}\frac{I^{\max(f,\tilde f)}_{[t+\delta, \infty)}(n)(c_i)}{D_{[t,t+\delta]}(n)(c_i)} = 0.\]
    If $\varepsilon \leq \delta$ we can set $\delta = \varepsilon$ and we are done. If $\varepsilon > \delta$, we have
    \begin{align*}
     D_{[t,t+\varepsilon]}(n)(c_i)& = 
      \int_{t}^{t+\delta}e^{-\frac{1}{2}c_i^2s+c_i b(s)}s^{n} (f(s) - \tilde f(s))\d s\\ 
     & \quad + \int_{t+\delta}^{t+\varepsilon}e^{-\frac{1}{2}c_i^2s+c_i b(s)}s^{n} (f(s) - \tilde f(s))\d s \\
     &  \geq \int_{t}^{t+\delta}e^{-\frac{1}{2}c_i^2s+c_i b(s)}s^{n} (f(s) - \tilde f(s))\d s\\
     &\quad - \int_{t+\delta}^{t+\varepsilon}e^{-\frac{1}{2}c_i^2s+c_i b(s)}s^{n} \max(f(s), \tilde f(s))\d s
    \end{align*}
    where the integrand of the first term is non-negative and for $c_i$ large enough the whole rhs term is non-negative. Using this inequality, we obtain 
    \begin{align*}
         \lim_{i\to \infty}\frac{|D_{[t+\varepsilon,\infty)}(n)(c_i)|}{D_{[t,t+\varepsilon]}(n)(c_i)} 
        \leq \lim_{i\to \infty}\frac{|D_{[t+\varepsilon,\infty)}(n)(c_i)|}{D_{[t,t+\delta]}(n)(c_i) - I^{\max(f,\tilde f)}_{[t+\delta,\infty)}(n)(c_i)}
         = 0.
    \end{align*}
    For $c$ large enough the denominator is positive, hence, it follows     
        \[\lim_{i\to \infty}\frac{D_{[t,\infty)}(n)(c_i)}{D_{[t,t+\varepsilon]}(n)(c_i)} = 1.\]
\end{proof}

\begin{lemma}\label{lem:hit2}
    If $I^{f}(c) = I^{\tilde f}(c)$ for all $c>C$ for some constant $C>0$, then
        \[I^{f}(n)(c) = I^{\tilde f}(n)(c),\]
    for all $n\in\nn_{0}$ and $c>C$ s.t.\ the integral exists.
\end{lemma}
\begin{proof}
The proof is by induction, i.e., we assume that
 \[\int_{0}^{\infty}e^{-\frac{c^2}{2}s + c b(s)}s^n f(s)\d s = \int_{0}^{\infty}e^{-\frac{c^2}{2}s + c b(s)}s^n \tilde f(s)\d s\]
 for all $c>C$. Multiplying both sides with $e^{-cb(0)}$ we obtain
  \[\int_{0}^{\infty}e^{-\frac{c^2}{2}s + c (b(s) -b(0))}s^n f(s)\d s = \int_{0}^{\infty}e^{-\frac{c^2}{2}s + c (b(s)-b(0))}s^n \tilde f(s)\d s.\]
  Taking derivatives w.r.t.\ $c$ on both sides, it follows
  \begin{align}
   & \int_{0}^{\infty}(-cs + b(s) -b(0))e^{-\frac{c^2}{2}s + c (b(s) -b(0))}s^n f(s)\d s \nonumber\\
   = & \int_{0}^{\infty}(-cs + b(s) -b(0))e^{-\frac{c^2}{2}s + c (b(s)-b(0))}s^n \tilde f(s)\d s.\label{eq:hitder1}
  \end{align}
We assume that $I^{f}(n+1) \neq I^{\tilde f}(n+1)$. Since $I^{f}(n+1)(c)$ is analytic, we find a sequence $c_i \to \infty$ such that $I^{f}(n+1)(c_i) \neq I^{\tilde f}(n+1)(c_i)$. Rearranging \eqref{eq:hitder1} this leads to
\begin{equation}\label{eq:hitci}
 c_i = \frac{\int_{0}^{\infty}e^{-\frac{c_i ^2}{2}s + c_i  (b(s) -b(0))}(b(s) -b(0))s^n (\tilde f(s)-f(s))\d s }{\int_{0}^{\infty}e^{-\frac{c_i ^2}{2}s + c_i  (b(s) -b(0))}s^{n+1} (f(s)-\tilde f(s))\d s}.
\end{equation}
Since $b$ is Lipschitz continuous, we find $L>0$ such that $|b(t) - b(0)|\leq L t$. Let us assume that there is $\varepsilon>0$ such that $f(t) = \tilde f(t)$ for $t\in (0,\varepsilon)$. Using Lemma \ref{lem:hit1}, we obtain 
\begin{align*}
 \infty & = \lim_{i\to\infty} c_i 
 =\lim_{i\to\infty} \frac{\int_{0}^{\infty}e^{-\frac{c_i ^2}{2}s + c_i  b(s)}(b(s) - b(0))s^n (\tilde f(s)-f(s))\d s }{\int_{0}^{\infty}e^{-\frac{c_i ^2}{2}s + c_i  b(s)}s^{n+1} (f(s)-\tilde f(s))\d s}\\
 & \leq \lim_{i\to\infty} \frac{\int_{0}^{\varepsilon}e^{-\frac{c_i ^2}{2}s + c_i  b(s))}L s^{n+1} (\tilde f(s)-f(s))\d s }{D(n+1)}\\
 &\quad +  \lim_{i\to\infty} \frac{\int_{\varepsilon}^{\infty}e^{-\frac{c_i ^2}{2}s + c_i  b(s))}\big((M+ms) s^{n} \big)(\tilde f(s)-f(s))\d s }{D(n+1)}\\
 &= L +0 <\infty.
\end{align*}
This is a contradiction to the assumption that $f\neq \tilde f$.
Let us now consider the case that $\tilde f = f$ in a neighbourhood of 0. We set $t^{\ast}: = \inf\{s\mid \tilde f(s)\neq f(s)\}$. 
Since $f$ and $\tilde f$ are continuous, we find sequences $t_j$ and $\varepsilon_j$ such that $t_j\to t^{\ast}$ and $f(t)\neq \tilde f(t)$ for $t\in [t_j, t_j+\varepsilon_j]$. 
In the same way as before it follows that
\begin{align*}
\lim_{i\to\infty} \frac{\int_{t_j}^{\infty}e^{-\frac{c_i ^2}{2}s + c_i  b(s)}(b(s) - b(t_j))s^n (\tilde f(s)-f(s))\d s }{\int_{t_j}^{\infty}e^{-\frac{c_i ^2}{2}s + c_i  b(s)}s^{n+1} (f(s)-\tilde f(s))\d s}
 &\leq L_j <\infty.
\end{align*}
By the dominated convergence theorem we find that
\begin{align*}
& \lim_{i\to\infty} \frac{\int_{t^{\ast}}^{\infty}e^{-\frac{c_i ^2}{2}s + c_i  b(s)}(b(s) - b(t^{\ast}))s^n (\tilde f(s)-f(s))\d s }{\int_{t^{\ast}}^{\infty}e^{-\frac{c_i ^2}{2}s + c_i  b(s)}s^{n+1} (f(s)-\tilde f(s))\d s}\\
 &\quad = \lim_{i\to\infty}\lim_{j\to\infty} \frac{\int_{t_j}^{\infty}e^{-\frac{c_i ^2}{2}s + c_i  b(s)}(b(s) - b(t_j))s^n (\tilde f(s)-f(s))\d s }{\int_{t_j}^{\infty}e^{-\frac{c_i ^2}{2}s + c_i  b(s)}s^{n+1} (f(s)-\tilde f(s))\d s}\\
 &\quad = \lim_{j\to\infty} L_j <\infty.
\end{align*}
Again, this is a contradiction to \eqref{eq:hitci} which concludes the proof.
  \end{proof}

We can now state the uniqueness theorem for \eqref{eq:hit1}.
\begin{theorem}\label{th:hit}
 Let $b$ be continuously differentiable with linear growth and $$-\infty <\liminf_{t\to 0}\frac{b(t) -b(0)}{t}\leq \limsup_{t\to 0}\frac{b(t) -b(0)}{t}<\infty.$$ If there exist $d \geq 0$ s.t.\ \eqref{eq:hit1} holds for all $c \geq d$ then the representation defines $f$ uniquely in the class of continuous functions.
\end{theorem}

\begin{proof}
Let $I^{f}(c) = I^{\tilde f}(c)$ for all $c>d$. By Lemma \ref{lem:hit2} we have $I^{f}(n)(c) = I^{\tilde f}(n)(c)$ for all $c>d$ and $n\in \rr$. We rewrite $I^{f}(n)(c)$ as
        \[I^{f}(n)(c) = \int_{0}^{\infty} e^{-s}s^{n}\mu(ds)
        \]
    with a measure
        $\mu =g \circ \lambda$
    where $g$ denotes the density function
        \[g(s) :=  e^{s}e^{-\frac{1}{2}c^2s+cb(s)}f(s)\]
    and $\lambda$  denotes the Lebesgue measure on $\rr$. A measure $\tilde\mu$ with density function $\tilde g$ is defined analogously via $\tilde f$.
    Recall that the Laguerre exponential polynomials $p(-s)e^{-s}$ lie dense in $L^{2}([0, \infty))$, see \cite[Lemma 1. (ii)]{ANG72}. 
    By Lemma \ref{lem:hit2} and linearity of the integral we have
        \[\int e^{s}q(s) g(s) \d s = \int e^{s}q(s) \tilde g(s) \d s\]
    for all polynomials $q$. Let now $h \in L^{2}([0, \infty))$ and $(q_n)$ be a sequence of polynomials with $e^{-s}q_n \overset{L^2} \to h$. 
    For $c$ large enough, $g$ and $\tilde g$ are positive, bounded and in $L^{2}$. It follows that
        \[\lim_{n\to\infty}\int e^{-s}q_n(s) g(s) \d s 
        = \int h(s) g(s) \d s \]
    and
        \[\lim_{n\to\infty}\int e^{-s}q_n(s) \tilde g(s) \d s 
        =\int h(s) \tilde g(s) \d s,\]
    hence,
        \[\int h(s) \mu(ds) = \int h(s) \tilde \mu(ds)\]
    for all $g \in L^{2}(B)$.\\
    If follows that $\mu = \tilde \mu$ a.e. Since $f$ and $\tilde f$ are continuous we have $f = \tilde f$.
\end{proof}

\begin{example}
 Let $b(t)\equiv b_0>0$. Then
 \[1 = \int_{0}^{\infty}e^{-\frac{c^2}{2}s + c b_0}f(s)\d s\]
 for all $c>0$. We set 
 \[f(s) = \frac{b_0}{\sqrt{2\pi s^3}}e^{-\frac{b_0^2}{2s}},\]
 then 
 \begin{align*}
   &\int_{0}^{\infty}\frac{b_0}{\sqrt{2\pi s^3}}e^{-\frac{c^2}{2}s + c b_0-\frac{b_0^2}{2s}}\d s
  = \int_{0}^{\infty}\frac{b_0}{\sqrt{2\pi s^3}}e^{-\frac{(b_0 - cs)^2}{2s}}\d s = 1,
 \end{align*}
 so $f$ indeed solves the problem. By Theorem \ref{th:hit} we can conclude that $f$ is the density function of the distribution of $\tau$, which is of course a well known result.
\end{example}

\section{Outlook and discussion}

A main contribution of this paper is that the FPT problem has been reduced to the solution of an integral equation because of the uniqueness result and thus the well known methods for such equations can be applied. The integral equation we are faced with here is a so-called Fredholm integral equation of the first kind. For these, numerical approaches suffer from the fact that these integral equations are usually ``ill-posed''. There are, however, quite a few approaches.  The book \cite{Waz11} provides a detailed overview. Among these methods are the method of regularization and the homotopy perturbation method. The more recent review article \cite{YZ19} presents even more methods such as the wavelet method. The paper \cite{CMW19} presents an algorithm to numerically solve Fredholm equations of the first kind and applies it to the FPT problem and the same integral equation as used in this paper (although the notation differs slightly). Due to the large number of results and methods, we also refer the interested reader to the literature cited in the above sources.

From a mathematical perspective, it seems natural to generalise this approach to the $n$-dimensional case, i.e., to the case of determining the distribution of the first passage time of a $n$-dimensional Brownian motion to some $n$-dimensional surface. As there is no canonical generalisation, we will present a possible approach in this section for future research. To this end, let $W_t = (W_t^{1},\ldots,W_t^{n})$ be an $n$-dimensional standard Brownian motion.
Let $d:\rr^{n} \to [0, \infty]$ be a continuous function such that $\tau :=\inf\{t\mid d(W_t)\geq t\}$ is an almost surely finite stopping time. A visualization of the setting can be seen in Figure \ref{fig:setting2}. We denote by $h:\rr^{n}\to \rr$ the density of the distribution of $W_{\tau}$. 
As above we decompose the transition kernel $p$ via first passage of $d$. Note that in contrast to the function $b$ above, $d$ maps from \textit{space} to \textit{time}. We have
\begin{align*}
p((0,0),(t,\textbf{x})) 
&= \int_{d^{-1}([0,t])}p\bigl((d(\textbf{y}),\textbf{y}),(t,\textbf{x})\bigr)P(W_\tau\in \d \textbf{y})  \\
&\quad+ P(W_t\in \d x, \tau>t)
\end{align*}
and 
\begin{align*}
1 
&= \frac{\int_{d^{-1}([0,t])}p((d(\textbf{y}),\textbf{y}),(t,\textbf{x})h(\textbf{y})\d \textbf{y}}{p((0,0),(t,\textbf{x}))}  
+ \frac{P(W_t\in \d x, \tau>t)}{p((0,0),(t,\textbf{x}))}.
\end{align*}
We set $\textbf{c}\in \rr^{n}$, $\textbf{x} = \textbf{c} t$ and let $t\to \infty$. 
As before the last term vanishes, if we choose $c$ appropriately.
We are left with
\begin{align*}\lim_{t\to \infty}\int_{d^{-1}([0,t])}\frac{p((d(\textbf{y}),\textbf{y}),(t,\textbf{c}t)}{p((0,0),(t,\textbf{c}t))}h(\textbf{y})\d \textbf{y} 
= \int_{{d^{-1}([0,\infty))}}e^{-\frac{\norm{\textbf{c}}^{2}}{2}d(\textbf{y}) + \textbf{c}\cdot \textbf{y}}h(\textbf{y})\d \textbf{y},
\end{align*}
where $\cdot$ denotes the standard scalar product and $\norm{~}$ the corresponding norm.
If $d(\textbf{y}) = \infty$ the integrand vanishes, so we can write the resulting equation as
\begin{equation}\label{eq:paramtization1}
 1 = \int_{\rr^{n}}e^{-\frac{\norm{\textbf{c}}^{2}}{2}d(\textbf{y}) + \textbf{c}\cdot \textbf{y}}h(\textbf{y})\d \textbf{y}
\end{equation}
for $\textbf{c}\in \rr^{n}$ such that the integral exists. The density $f$ of the distribution of $\tau$ can be obtained via
\[f(t) = \int_{h^{-1}(t)}h(\textbf{y})\d \sigma\]
where $\d \sigma$ denotes integration over the (hyper) surface $h^{-1}(t)$. 
The proof of uniqueness seems to us to be possible in principle under suitable additional preconditions with the methods presented here, but it is beyond the scope of this paper.

    \begin{figure}[htbp]
        \centering
        \includegraphics[width=0.7\textwidth]{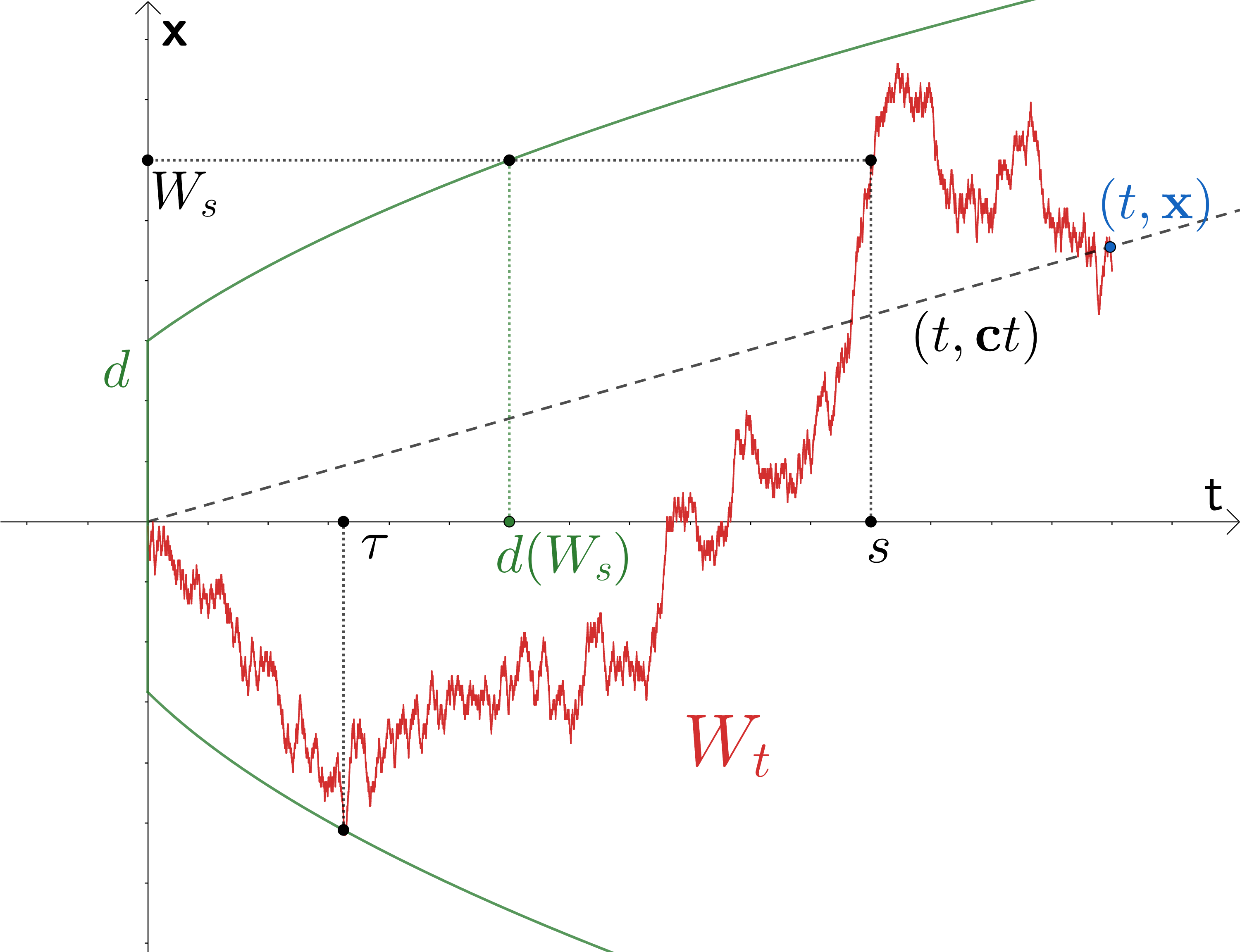}
        \caption{ One-dimensional visualization of the parametrization via $\textbf{x}$}
        \label{fig:setting2}
    \end{figure}



\end{document}